\documentclass[12pt,reqno]{article}
\usepackage[english]{babel}
\usepackage{amsmath,amsthm}
\usepackage{amsfonts}
\usepackage{natbib}
\usepackage{graphicx}
\usepackage{enumerate}
\usepackage[usenames,dvipsnames]{color}
\usepackage{subfigure}
\usepackage[right,pagewise,displaymath, mathlines]{lineno}
\usepackage{epstopdf}
\usepackage{color}
\usepackage{multirow}

\numberwithin{equation}{section}
\numberwithin{figure}{section}
\numberwithin{table}{section}

\newtheorem{theorem}{Theorem}[section]

\newtheorem{lemma}{Lemma}[section]

\numberwithin{equation}{section}

\begin{document}

\begin{center}

{\Large\bf Statistical foundations for assessing the difference between the classical and weighted-Gini betas}

\vspace*{7mm}

{\large Nadezhda Gribkova}

\medskip

\textit{Department of Probability and Mathematical Statistics, St.\,Petersburg State University,
\break St.\,Petersburg 199034, Russia}

\bigskip

{\large Ri\v cardas Zitikis}

\medskip

\textit{School of Mathematical and Statistical Sciences,
Western University, \break London, Ontario N6A 5B7, Canada}

\end{center}

\medskip

\noindent
\textbf{Abstract.}
The `beta' is one of the key quantities in the capital asset pricing model (CAPM). In statistical language, the beta can be viewed as the slope of the regression line fitted to financial returns on the market against the returns on the asset under consideration. The insurance counterpart of CAPM, called the weighted insurance pricing model (WIPM), gives rise to the so-called weighted-Gini beta. The aforementioned two betas may or may not coincide, depending on the form of the underlying regression function, and this has profound implications when designing portfolios and allocating risk capital. To facilitate these tasks, in this paper we develop large-sample statistical inference results that, in a straightforward fashion, imply confidence intervals for, and hypothesis tests about, the equality of the two betas.

\medskip



\newpage

\section{Introduction}
\label{imtro}

The `beta' is one of those classical quantities that we find in virtually every financial engineering text that discusses topics related to the capital asset pricing model (CAPM). For the state-of-the-art appraisal of the area with extensive references, we refer to Levy (2010, 2011). Statistically speaking, the beta is the slope
\begin{equation}
\label{beta}
\beta=\frac{\textbf{Cov}[X,Y]}{\textbf{Cov}[Y,Y]}
\end{equation}
of the least squares regression line, with $X$ denoting the response or dependent variable (such as the return on an asset) and $Y$ is the explanatory or independent variable (such as the return on the market). Of course, we assume that the variance $\textbf{Var}[Y]:=\textbf{Cov}[Y,Y]$ of $Y$ is finite and non-zero, and the covariance $\textbf{Cov}[X,Y]$ between $X$ and $Y$ is finite. (We use `$:=$' when emphasizing that some equalities hold by definition.)

An extension of the CAPM to insurance has turned out to be a complex task due to issues such as skewness and heavy tails of the underlying random variables. Furman and Zitikis (2017) have put forward arguments showing that the task is feasible, and their proposed solution hinges, in part, on the so-called weighted-Gini beta
\begin{equation}
\label{beta_G}
\beta_G=\frac{\textbf{Cov}[X,w\circ F_Y(Y)]}{\textbf{Cov}[Y,w\circ F_Y(Y)]},
\end{equation}
where $F_Y$ is the cumulative distribution function (cdf) of $Y$, $w: [0,1]\to [-\infty,\infty]$ is a weight function (to be discussed below), and $w\circ F_Y$ is the composition of $w$ and $F_Y$. Naturally, we assume the obvious conditions that make the beta well-defined and finite, and we also assume throughout the paper that $w$ is finite on the open interval $(0,1)$. Next are several illustrative examples of $w$ that we use to justify certain technical intricacies employed in this paper.

Examples that we find in the literature usually deal with non-decreasing functions $w$, chosen either based on some regulatory frameworks such as Basel Accords for Banking (e.g., Sawyer, 2012) and Solvency for Insurance (e.g., Sandstr\"{o}m, 2010), or based on other considerations such as economic axioms (e.g., von Neumann and Morgenstern, 1944; Quiggin, 1993; and references therein). For example, the weight function $w_{\mathrm{PHT}}(t)=\nu(1-t)^{\nu-1}$, with PHT standing for the `proportional hazards transform,' has arisen in Insurance when $\nu\in (0,1]$ (Wang, 1995) and Econometrics  when $\nu \ge 1$ (Donaldson and  Weymark, 1980; Kakwani, 1980; see also Zitikis and Gastwirth, 2002; and references therein). The function $w_{\mathrm{CTE}}(t)=\mathbf{1}\{t>\nu \}$ for various parameter values
$\nu\in (0,1)$ arises in contexts associated with the conditional tail expectation (e.g., McNeil et al., 2015; and references therein). In terms of mathematical properties, these functions are quite different: $w_{\mathrm{PHT}}$ is  unbounded for every $\nu\in (0,1)$, continuous on the compact interval $[0,1]$ for every $\nu\ge 1$, whereas the function $w_{\mathrm{CTE}}$ is discontinuous, though bounded, for every $\nu\in (0,1)$. We shall see later in this paper that these features place considerable constraints on the technical apparatus that we can employ, thus occasionally requiring involved arguments in order to accommodate cases such as the aforementioned examples of $w$.

We have organized the rest of the paper as follows. In Section \ref{sect-2} we construct empirical estimators for the two aforementioned betas and discuss their consistency under several complementary sets of conditions. Then we establish asymptotic normality of the difference between the two estimators. Proofs are in Section \ref{sect-4}, and concluding remarks make up Section \ref{summary}. Throughout the paper, we use the following two functions prominently: the conditional-mean function
\[
g_{X\mid Y}(y)=\mathbf{E}[X \mid Y=y]
\]
and the conditional-variance function
\[
v^2_{X\mid Y}(y)=\textbf{Var}[X \mid Y=y],
\]
both defined on the range of $Y$ values. Furthermore, we use $\stackrel{a.s.}{\to}$ to denote convergence almost surely, $\stackrel{\mathbf{P}}{\to}$ convergence in probability, and $\stackrel{law}{\to}$ convergence in law/distribution. We use $c$ to denote various finite constants whose values usually change from line to line.

\section{Main results}
\label{sect-2}

Coming back to definition (\ref{beta_G}) of $\beta_G$, we see that  when the conditional-mean function $g_{X\mid Y}(y)$ is linear on the range of $Y$ values, that is, $g_{X\mid Y}(y)=\alpha + \gamma y$ for some $\alpha, \gamma \in \mathbb{R}$, then $\beta_G=\beta ~(=\gamma)$ irrespective of the weight function $w$. This implies that the weighted insurance pricing model (WIPM) collapses into the classical CAPM, and this feature has been pointed out and utilized by Furman and Zitikis (2010, 2017). For example, in the bivariate Gaussian case, the function $g_{X\mid Y}(y)$ is linear, with the slope equal to $\beta $ given by equation (\ref{beta}). The bivariate elliptical distribution also has a linear regression function (e.g., Furman and Zitikis, 2008, 2017), and so do several bivariate Pareto distributions, though of course not all of them. For details and examples, we refer to Su (2016), Su and Furman (2017), and references therein.

If the linearity of $g_{X\mid Y}(y)$ does not hold, then how far can $\beta_G$ be from $\beta$? This is important to know because if the difference between $\beta_G$ and $\beta$ is not practically significant, even when it is not actually zero, then we can safely ignore the difference and work with the classical beta, for which statistical inference has been well developed in various contexts. This brings us to the main topic of the present paper, namely, the development of statistical tests for assessing the magnitude of
\[
\Delta=\beta_G-\beta.
\]
For this task, we need an empirical estimator for $\Delta $. Let $(X_i,Y_i)$, $i=1,2\dots$, be independent copies of the random pair  $(X,Y)$.  For every integer $n \geq 1$, let $\widehat{F}_Y$ be the empirical estimator of the cdf $F_Y$ defined by
\begin{equation}
\label{new_ecdf}
\widehat{F}_Y (y) = \frac 1{n + 1} \sum_{k=1}^n \mathbf{1}\{Y_k \leq y\},
\end{equation}
where $\mathbf{1}\{Y_k \leq y\}$ is the indicator of the event $\{Y_k \leq y\}$. Note that this empirical estimator slightly differs from the classical empirical cdf because it employs jumps of the size $1/(n+1)$, instead of the usual $1/n$. This adjustment, being not an issue from the asymptotic point of view, is necessary to ensure that all the values of the empirical cdf are located inside the open interval $(0,1)$ on which the weight function $w$ is finite. We are now ready to introduce an empirical estimator of $\Delta$, which is
\begin{equation}
\label{plug_in}
\widehat{\Delta}_n=\widehat{\beta}_{G,n}-\widehat{\beta}_n,
\end{equation}
where
\begin{equation*}
\widehat{\beta}_{G,n} =\frac{\sum_{k=1}^{n}(X_k-\overline{X}) (w\circ \widehat{F}_Y(Y_k)-\overline{z}_0) }{\sum_{k=1}^{n}(Y_k-\overline{Y}) (w\circ \widehat{F}_Y(Y_k)-\overline{z}_0)  }
\end{equation*}
and
\begin{equation*}
\widehat{\beta}_n=\frac{\sum_{k=1}^{n}(X_k-\overline{X}) (Y_k-\overline{Y}) }{\sum_{k=1}^{n}(Y_k-\overline{Y})^2 }
\end{equation*}
with $\overline{X}=n^{-1}\sum_{k=1}^{n}X_k$, $\overline{Y}=n^{-1}\sum_{k=1}^{n}Y_k$, and
$\overline{z}_0=n^{-1}\sum_{k=1}^{n} w\circ \widehat{F}_Y(Y_k)$.

\begin{theorem}
\label{theorem-1}
If the expectations  $\mathbf{E}[X]$, $\mathbf{E}[Y^2]$, and $\mathbf{E}[XY]$ are finite, and the weight function $w$ is continuous on $[0,1]$, then, when $n\to\infty$,
\[
\widehat{\Delta}_n \stackrel{a.s.}{\to} \Delta .
\]
\end{theorem}

Though simple and elegant, Theorem \ref{theorem-1} does not cover a number of important cases. To see this, we recall the earlier noted weight functions: when $\nu\ge 1$, the function $w_{\mathrm{PHT}}(t)=\nu(1-t)^{\nu-1}$ is continuous on the compact interval $[0,1]$ and thus Theorem \ref{theorem-1} is applicable, but the case $\nu \in (0,1)$, which is of particular interest in Insurance, produces unbounded  $w_{\mathrm{PHT}}$ on $[0,1]$. Thus, we cannot apply Theorem \ref{theorem-1} in the latter case, nor can we apply the theorem in the case of $w_{\mathrm{CTE}}(t)=\mathbf{1}\{t>\nu \}$, which is discontinuous for every $\nu \in (0,1)$. Our next theorem is designed to accommodate cases such as these.

We use $L_q$, $1\leq q\leq \infty$, to denote the set of all Borel-measurable functions $f: [0,1]\to \mathbb{R}$ such that $\|f\|_q :=(\int_0^1 |f|^q\mathrm{d}\lambda)^{1/q}<\infty $ when $1\leq q<\infty$, and $\|f\|_{\infty}:=\mathrm{ess}\sup_{t\in [0,1]}|f(t)|<\infty $ when $q=\infty$, where $\lambda $ is the Lebesgue measure. The quantile function of $Y$ is defined by $F_Y^{-1}(t)= \inf \{ y: F_Y(y) \geq t \}$ for all $t\in (0,1]$.

\begin{theorem}
\label{theorem-2}
Assume that the expectation $\mathbf{E}[Y^2]$ is finite, and the cdf $F_Y$ is continuous. If $\mathbf{E}[(\mathbf{E}[X^2\mid Y])^p]<\infty $ and $w^2\in L_q$ for some $p,q \in [1, \infty ]$ such that $p^{-1}+q^{-1}=1$, then, when $n\to\infty$,
\[
\widehat{\Delta}_n \stackrel{\mathbf{P}}{\rightarrow} \Delta .
\]
\end{theorem}

We note at the outset that the conditions on $X$ and $Y$ in Theorem \ref{theorem-2} are stronger than those in Theorem \ref{theorem-1}, and this is so in order to weaken conditions on the weight function $w$. To see how the two sets of conditions are related, we first note that since $p\ge 1$, the condition $\mathbf{E}[(\mathbf{E}[X^2\mid Y])^p]<\infty $ implies   $\mathbf{E}[X^2]<\infty $, and since we assume $\mathbf{E}[Y^2]$, we have $\mathbf{E}[|XY|]<\infty $, which is one of the moment conditions in Theorem \ref{theorem-1}. Hence, the requirements on the random variables have indeed increased, but very importantly, the requirements on the weight function $w$ in Theorem \ref{theorem-2} have decreased considerably. To see the benefits, we next discuss two extreme cases covered by Theorem \ref{theorem-2}.

First, when $p=1$, we have $q=\infty $, and thus the weight function $w$ must be bounded. This case covers the function $w_{\mathrm{PHT}}(t)=\nu(1-t)^{\nu-1}$ when $\nu\ge 1$, as well as the discontinuous function $w_{\mathrm{CTE}}(t)=\mathbf{1}\{t>\nu \}$. Second, when $p=\infty$, the condition $\mathbf{E}[(\mathbf{E}[X^2\mid Y])^p]<\infty $ means that the conditional-variance function $v^2_{X\mid Y}$ must be bounded. Of course, when $p=1$, then $q=1$ and thus $w^2$ is integrable, which is a very mild assumption on the weight function $w$: the CTE weight function $w_{\mathrm{CTE}}$ is always such, whereas the PHT weight function $w_{\mathrm{PHT}}$ satisfies the requirement when $\nu>1/2$. The latter restriction appears naturally when considering statistical inference for the PHT risk measure (e.g., Jones and Zitikis, 2007; Necir et al. (2007); Necir and Meraghni (2009);  Brahimi et al., 2011; and references therein).

Note also that in Theorem \ref{theorem-2} we assume continuity of $F_Y$ which, though possibly restrictive in some applications, brings tangible benefits into the development of statistical inference. For example, the earlier expression for estimator $\widehat{\beta}_{G,n}$ turns into the following easier manageable expression
\begin{equation}
\label{plug_G}
\widehat{\beta}_{G,n} =\frac{\sum_{k=1}^{n}(X_{[k:n]}-\overline{X}) (w_{k,n}-\overline{z}_0) }{\sum_{k=1}^{n}(Y_{k:n}-\overline{Y}) (w_{k,n}-\overline{z}_0)  } ,
\end{equation}
where
\begin{itemize}
\item
$Y_{1:n}, \dots, Y_{n:n}$ are the (ascending) order statistics of $Y_1, \dots, Y_n$;
\item
$X_{[1:n]}, \dots , X_{[n:n]}$ are the induced order statistics corresponding to $Y_{1:n}, \dots , Y_{n:n}$;
\item
$\overline{z}_0=n^{-1}\sum_{k=1}^{n}w_{k,n}$
with $w_{k,n}= w(k/(n+1))$.
\end{itemize}

Unlike in Theorem \ref{theorem-1}, where we established strong consistency, in Theorem \ref{theorem-2} we deal with (weak) consistency. This shift from strong to weak consistency puts us firmly on the practical path and leads to attractive and highly encompassing conditions, as we have already seen from the two extreme cases analyzed above. On the other hand, when establishing asymptotic normality, which is of our primary interest in the current paper and makes up the contents of the next theorem, we rely only on consistency, and thus our particular focus on this mode of convergence. We use the notations $z_0=\mathbf{E}[w\circ F_Y(Y)]$, $B=\mathbf{Cov}[Y,w\circ F_Y(Y)]$, and  $D=\mathbf{Var}[Y]$.

\begin{theorem}
Let the cdf $F_Y$ be continuous, and let the weight function $w$ be  continuously differentiable except possibly at a finite number of points at which $F^{-1}_Y$ and $g_{X\mid Y}\circ F^{-1}_Y$ must be continuous. Furthermore, assume that there is $b\in [0,1)$ such that the following three assumptions hold:
\begin{enumerate}[{\rm (i)}]
\item  \label{as-1a}
there is $c<\infty $ such that, for all $t\in (0,1)$ where $w'(t)$ exists,
\[
|w(t)|, ~ t(1-t)|w'(t)| \leq c \big ( t(1-t) \big )^{-b/2} ;
\]
\item  \label{as-1b}
the moment $\mathbf{E}[|Y|^{r_1}] $ is finite for some    $r_1>r:=\max\{4,2/(1-b)\}$;
\item \label{as-1c}
there are $\epsilon>0$ and $c<\infty $ such that, for all $t\in (0,1)$,
\[
v^2_{X\mid Y}\circ F_Y^{-1}(t)\leq  c \big ( t(1-t) \big )^{-2/r+\epsilon},
\]
where we have $2/r=\min\{1/2, 1-b\}$.
\end{enumerate}
Then, when $n \to \infty$,
\[
n^{1/2}(\widehat{\Delta}_n-\Delta) \stackrel{law}{\to}\mathcal{N}(0, \Upsilon_1^2+\Upsilon^2_{2}),
\]
where
\begin{equation}
\label{sigma_1main}
\Upsilon_1^2=\int_0^1v^2_{X\mid Y}\circ F_Y^{-1}(t) \bigg( {w(t)-z_0\over B} -{F_Y^{-1}(t)-\mathbf{E}[Y] \over D} \bigg)^2\, \mathrm{d}t
\end{equation}
and
\begin{align}
\Upsilon_2^2
&={1\over B^2} \int_0^1 \int_0^1 \big (w(s)-z_0\big )\big (w(t)-z_0\big )\big (\min(s,t) -st\big ) \,\mathrm{d}H_1(s)\,\mathrm{d}H_1(t)
\notag
\\
&\qquad -{2\over BD} \int_0^1 \int_0^1 \big (w(s)-z_0\big )\big (\min(s,t) -st\big )\,\mathrm{d}H_1(s)\,\mathrm{d}H_2(t)
\notag
\\
&\qquad \qquad +{1\over D^2}\int_0^1 \int_0^1 \big (\min(s,t) -st\big )\,\mathrm{d} H_2(s)\,\mathrm{d}H_2(t)
\label{sigma_2main}
\end{align}
with the functions
\[
H_1(t) =\big ( g_{X\mid Y}\circ F_Y^{-1}(t) -\mathbf{E}[X] \big )-\beta_G \big ( F_Y^{-1}(t)-\mathbf{E}[Y] \big )
\]
and
\[
H_2(t) =\big (g_{X\mid Y}\circ F_Y^{-1}(t)-\mathbf{E}[X] \big )\big ( F_Y^{-1}(t)-\mathbf{E}[Y] \big ) -\beta \big ( F_Y^{-1}(t)-\mathbf{E}[Y] \big )^2.
\]
\label{thm_2main}
\end{theorem}

The conditions of Theorem~\ref{thm_2main} are, naturally, stronger than those of Theorem~\ref{theorem-2}. To gain more intuition on them, we next show how the two sets of conditions are related to each other. First, both theorems require continuity of $F_Y$. Furthermore, Theorem~\ref{thm_2main} requires at least the fourth finite moment of $Y$, whereas Theorem~\ref{theorem-2} requires at least the second finite moment of $Y$. This is natural because in the two theorems we deal with consistency and weak convergence of an empirical estimator of the second moment $\mathbf{E}[Y^2]$, which is an integral part of the definition of the classical beta.

To show that the conditions  $\mathbf{E}[(\mathbf{E}[X^2\mid Y])^p]<\infty $ and $w^2\in L_q$ of Theorem~\ref{theorem-2} are implied by conditions (\ref{as-1b}) and (\ref{as-1c}) of Theorem~\ref{thm_2main}, we first note that $\mathbf{E}[(\mathbf{E}[X^2\mid Y])^p]<\infty $ is equivalent to $v^2_{X\mid Y}\circ F_Y^{-1}\in L_p$. It is now easy to check that, for every $b\in [0,1)$, if we set $p=(1+\delta)/(1-b)$ with any  $\delta >0$ such that $\delta<\epsilon/(1-b-\epsilon)$ (we can always assume $\epsilon <1-b$ without loss of generality), then the two conditions of Theorem~\ref{theorem-2} are satisfied.

Furthermore, note that the definition of $\Upsilon_2^2$ implicitly requires that the functions $H_1$ and $H_2$ should be properly defined, which is the case whenever $g_{X\mid Y}\circ F_Y^{-1}$ is in the class of functions of bounded variation on the interval $[\epsilon,1-\epsilon]$ for every $\epsilon>0$. Assumptions (\ref{as-1a})--(\ref{as-1c}) imply that both $\Upsilon_1^2$ and $\Upsilon_2^2$ are finite.

As an illustrative example of Theorem~\ref{thm_2main}, consider the bivariate Gaussian distribution, in which case $g_{X\mid Y}(y)=\mathbf{E}[X] + \beta (y-\mathbf{E}[Y] )$ with the slope $\beta $ defined by equation (\ref{beta}). The corresponding conditional-variance function is $v^2_{X\mid Y}(y)=(1-\rho^2)\textbf{Var}[X]$,
where $\rho =\mathbf{Corr}[X,Y]$ is the correlation coefficient. With the notation $C=\mathbf{Cov}[X,Y]$, we have $v^2_{X\mid Y}(y)=(1-\rho^2)C^2/(D\rho^2)$ and thus
\begin{equation}
\label{sigma_1main-gauss}
\Upsilon_1^2= \bigg ({1\over \rho^2}-1 \bigg ){C^2 \over D} \int_0^1 \bigg( {w(t)-z_0\over B} -{\Phi^{-1}(t) \over \sqrt{D} }  \bigg)^2\, \mathrm{d}t ,
\end{equation}
where $\Phi$ is the standard normal cdf. Since $g_{X\mid Y}(y)$ is linear and $\beta_G=\beta$, we have $\Upsilon_2^2=0$. Consequently, Theorem \ref{thm_2main} says that the asymptotic distribution of $n^{1/2}(\widehat{\Delta}_n-\Delta )$ is $\mathcal{N}(0,\Upsilon_1^2)$.

To construct an empirical estimator for $\Upsilon_1^2$ when it is given by formula (\ref{sigma_1main-gauss}) is not a complex task, as we only need to estimate $\rho$, $B$, $C$, $D$, and $z_0$, which are fairly straightforward tasks. However, Theorem \ref{thm_2main} is not about the bivariate normal distribution -- it is about estimating $\Delta$ when no specific bivariate distribution is assumed. Hence, we need to know the critical values upon which confidence intervals and hypothesis tests would be based, and this requires empirical estimators for $\Upsilon_1^2$ and $\Upsilon_2^2$ defined by equations (\ref{sigma_1main}) and (\ref{sigma_2main}). Though doable with the help of $L$-statistics, it turns out to be a messy task. This prompts us to think of another method for estimating the critical values, and bootstrap is an attractive option (e.g., Chernick and LaBudde, 2011; and references therein). It is well known, however, that bootstrap may not always work (e.g., Athreya, 1987; Bickel et al, 1997; Hall, 1992; Mammen, 1992), but when the underlying asymptotic normality is established (e.g., Hall, 1992; Mammen, 1992), the bootstrap does work. This reveals the value of Theorem \ref{thm_2main} even when its direct use for producing statistical inference has been circumvented by bootstrap, either naive or more advanced, like for example ``$m$ out of $n$'' as in Bickel et al. (1997); see also Gribkova and  Helmers (2007, 2011), and references therein.

We conclude this section with the note that the function $g_{X\mid Y}\circ F_Y^{-1}(t)$ is known in the literature as the quantile-regression function, which also gives rise to the cumulative quantile-regression function $\int_0^u g_{X\mid Y}\circ F_Y^{-1}(t)\mathrm{d} t$. Estimation of these functions in the context of empirical processes was initiated by Rao and Zhao (1995), and then taken over by Tse (2009), who in a series of papers has developed a wide-ranging statistical inference theory. We refer to Tse (2015) for details and further references on the topic. The quantile conditional-variance function $v^2_{X\mid Y}\circ F_Y^{-1}(t)$ also plays a prominent role in the aforementioned works, as it does in the present paper as well.

\section{Proofs}
\label{sect-4}

In this section we prove all the three theorems formulated above.

\subsection{Proof of Theorem \ref{theorem-1}}

Due to the assumed finiteness of moments, we have $\widehat{\beta}_n \stackrel{a.s.}{\to} \beta$  when $n \to \infty$, and so the theorem follows provided that
\begin{equation}
\label{s1}
\overline{z}_0 \stackrel{a.s.}{\to} z_0:=\mathbf{E}[w\circ F_Y(Y)],
\end{equation}
\begin{equation}
\label{s2}
{1\over n}\sum_{k=1}^{n}Y_k w\circ \widehat{F}_Y(Y_k)\stackrel{a.s.}{\to} \mathbf{E}[Yw\circ F_Y(Y)],
\end{equation}
\begin{equation}
\label{s3}
{1\over n}\sum_{k=1}^{n}X_k w\circ \widehat{F}_Y(Y_k)\stackrel{a.s.}{\to} \mathbf{E}[Xw\circ F_Y(Y)].
\end{equation}
Note that all the expectations on the right-hand sides of the above three statements are finite because the weight function $w$ is bounded and the moments $\mathbf{E}[X]$ and  $\mathbf{E}[Y]$ are finite.

In order to prove statement \eqref{s1}, we write
\begin{equation}
\label{proof_s1_2}
\overline{z}_0=  {1\over n}\sum_{k=1}^{n} w\circ F_Y(Y_k) +{1\over n}\sum_{k=1}^{n}\big ( w\circ \widehat{F}_Y(Y_k) - w\circ F_Y(Y_k)\big ).
\end{equation}
The first average on the right-hand side of equation \eqref{proof_s1_2} converges to $z_0$ almost surely by the classical strong law of large numbers, and the absolute value of the second average does not exceed $\sup_{y\in \mathbb{R}}|w\circ \widehat{F}_Y(y)-w\circ F_Y(y)|$, which converges to zero almost surely because of the uniform continuity of $w$ (since it is continuous on the compact interval $[0,1]$) and the classical Glivenko-Cantelli Theorem. This establishes statement \eqref{s1}.

To prove statement \eqref{s2}, we write
 \begin{equation}
\label{s2_p}
{1\over n}\sum_{k=1}^{n}Y_k w\circ \widehat{F}_Y(Y_k)={1\over n}\sum_{k=1}^{n}Y_k w\circ F_Y(Y_k)+{1\over n}\sum_{k=1}^{n}Y_k \big ( w\circ \widehat{F}_Y(Y_k)-w\circ F_Y(Y_k) \big ).
\end{equation}
By the strong law of large numbers, the first average on the right-hand side of equation \eqref{s2_p} converges to $\mathbf{E}[Yw\circ F_Y(Y)]$ almost surely. We are left to show that the second average converges to $0$ almost surely. To this end, we estimate its absolute value from above by
\begin{equation}
\label{s2_p3}
\bigg ( {1\over n}\sum_{k=1}^{n}|Y_k| \bigg ) \sup_{y\in \mathbb{R}}\big |w\circ \widehat{F}_Y(y)-w\circ F_Y(y)\big |.
\end{equation}
By the classical strong law of large numbers, $n^{-1}\sum_{k=1}^{n}|Y_k|$ converges almost surely to $ \mathbf{E}[|Y|]$, which is finite. As already noted above, the supremum on the right-hand side of bound (\ref{s2_p3}) converges to $0$ almost surely. All these facts establish statement \eqref{s2}. The proof of statement \eqref{s3} is virtually identical and thus omitted. Theorem~\ref{theorem-1} follows.

\subsection{Proof of Theorem \ref{theorem-2}}

We reduced the proof of Theorem~\ref{theorem-1} to verifying statements \eqref{s1}--\eqref{s3}. Now we do exactly the same but instead of proving the three statements almost surely, we prove them `in probability.' Note at the outset that since $F_Y$ is continuous, we can replace  $\widehat{F}_Y(Y_{k:n})$ by $k/(n+1)$ and thus both $\widehat{\beta}_{G,n}$ and $\overline{z}_0$ are the same as in equation (\ref{plug_G}). The proof of statement \eqref{s1} is simple: since $w$ is integrable, $\overline{z}_0$ converges to $\int_0^1 w\mathrm{d}\lambda$ when $n\to \infty $, and the latter integral is equal to $\mathbf{E}[w\circ F_Y(Y)]$, which is $z_0$. Statement \eqref{s1} follows.

To prove statement \eqref{s2}, we note that the quantity on the left-hand side is equal to  $n^{-1}\sum_{k=1}^{n}Y_{k:n} w_{k,n}$  almost surely. Since $\mathbf{E}[Y^2]$ is finite, we have $F_Y^{-1}\in L_2$, and we also have $w\in L_2$ because $w^2\in L_q$ for some $q \ge 1$. Hence, statement \eqref{s2}  follows from the strong law of large numbers for $L$-statistics (van Zwet, 1980).

To prove the in-probability version of statement \eqref{s3}, we write
\[
{1\over n}\sum_{k=1}^{n}X_{[k:n]} w_{k,n}={1\over n}T_{n,1}+{1\over n}T_{n,2},
\]
where
\[
T_{n,1}=\sum_{k=1}^{n}g_{X\mid Y}(Y_{k:n}) w_{k,n}
\]
and
\[
T_{n,2}= \sum_{k=1}^{n}\big (X_{[k:n]} -g_{X\mid Y}(Y_{k:n})\big ) w_{k,n} .
\]
We have $\mathbf{E}[(\mathbf{E}[X^2\mid Y])^p]<\infty $ for some $p\ge 1$, and thus $g_{X\mid Y}\circ F_Y^{-1}\in L_2$. This allows us to use the strong law of large numbers for $L$-statistics (van Zwet, 1980;  Corollary~2.1), which implies
\[
{1\over n} T_{n,1} \stackrel{a.s.}{\to} \int_0^1 g_{X\mid Y}\circ F_Y^{-1}w\mathrm{d}\lambda ,
\]
where the integral on the right-hand side is equal to $\mathbf{E}[Xw\circ F_Y(Y)]$. Hence, statement \eqref{s3} follows if
$n^{-1}T_{n,2} \stackrel{\mathbf{P}}{\to} 0$ when $n \to \infty$, which means that, for every $\varepsilon>0$,
\begin{equation}
\label{thm2_p2b}
\mathbf{P}\big(|T_{n,2}| >n\varepsilon \big) \rightarrow 0.
\end{equation}
To prove this statement, we recall (Bhattacharya, 1974) that the induced order statistics $X_{[1:n]}, \dots , X_{[n:n]}$  are conditionally independent, given  $Y_{1:n}, \dots, Y_{n:n}$, and follow the conditional cdf's $F(x\mid Y_{1:n}), \dots , F(x\mid Y_{n:n})$, respectively, where
$F(x\mid y)=\textbf{P}[X\leq x|Y=y]$. An application  of Markov's inequality yields
\begin{align}
\mathbf{P}\big(|T_{n,2}| >n\varepsilon \big)
&=\mathbf{E}\bigg [\mathbf{P}\Big(\Big|\sum_{k=1}^{n}\big (X_{[k:n]} -g_{X\mid Y}(Y_{k:n})\big ) w_{k,n}\Big| >n\varepsilon \Big| Y_1,\dots,Y_n\Big) \bigg ]
\notag
\\
&\le \frac{1}{n^2\varepsilon^2}\sum_{k=1}^{n}\mathbf{E}\Big [\mathbf{E} \Big [\big (X_{[k:n]} -g_{X\mid Y}(Y_{k:n})\big )^2\Big| Y_1,\dots,Y_n\Big] \Big ] w^2_{k,n}
\notag
\\
&= \frac{1}{n^2\varepsilon^2}\sum_{k=1}^{n}\mathbf{E}\big [ v^2_{X\mid Y}(Y_{k:n})\big] w^2_{k,n} .
\label{thm2_p3b}
\end{align}
Next we apply H\"{o}lder's inequality on the right-hand side of bound \eqref{thm2_p3b} and obtain
\begin{align}
\mathbf{P}\big(|T_{n,2}| >n\varepsilon \big)
&\le \frac{1}{n\varepsilon^2}
\bigg(\frac{1}{n}\sum_{k=1}^{n}\Big ( \mathbf{E}\big [ v^2_{X\mid Y}(Y_{k:n})\big]\Big )^p\bigg)^{1/p}  \bigg({1\over n} \sum_{k=1}^{n} |w_{k,n}|^{2q}\bigg)^{1/q}
\notag
\\
&\le  \frac{1}{n\varepsilon^2}
\bigg(\frac{1}{n}\sum_{k=1}^{n}\mathbf{E}\big [ v^{2p}_{X\mid Y}(Y_{k:n})\big]\bigg)^{1/p}  \bigg({1\over n} \sum_{k=1}^{n} |w_{k,n}|^{2q}\bigg)^{1/q}
\notag
\\
&=  \frac{1}{n\varepsilon^2}
\bigg(\frac{1}{n}\sum_{k=1}^{n}\mathbf{E}\big [ v^{2p}_{X\mid Y}\circ F^{-1}_Y (U_{k})\big]\bigg)^{1/p}  \bigg({1\over n} \sum_{k=1}^{n} |w_{k,n}|^{2q}\bigg)^{1/q},
\label{thm2_p5b}
\end{align}
where $U_1,\dots,U_n$ are independent $(0,1)$-uniform random variables.
By the classical law of large number, the first sum on the right-hand side of equation (\ref{thm2_p5b}) convergence to $\int_0^1 v^{2p}_{X\mid Y}\circ F^{-1}_Y (t)\,\mathrm{d}t$, whereas the second sum converges to $\int_0^1 |w(t)|^{2q}\mathrm{d}t$. Both integrals are finite by assumption, and thus statement~\eqref{thm2_p2b} follows. This completes the proof of Theorem \ref{theorem-2}.
\hfill $\Box$

\subsection{Proof of Theorem \ref{thm_2main}}

The proof of Theorem \ref{thm_2main} is involved, and we thus carry it out in several steps. First we show that when deriving the limit distribution of $n^{1/2}(\widehat{\Delta}_n-\Delta)$ we can restrict ourselves to zero-centered variables. This considerably simplifies our following considerations.

\subsubsection{From general to centered rv's}
\label{sect-4transition}

The centered versions of the random variables $X$ and $Y$ are
$\xi =X-\mathbf{E}[X] $ and $\eta=Y-\mathbf{E}[Y] $, respectively, and the centered version of the weight function $w$ is $ w_0(t)=w(t)-z_0 $ with $ z_0=\mathbf{E}[w\circ F_Y(Y)]$. Note that $F_Y(Y)$ is equal to $F_{\eta}(\eta)$ and hence
\begin{equation*}
\Delta
={\mathbf{E}[\xi  w_0\circ F_{\eta}(\eta) ] \over
\mathbf{E}[\eta w_0\circ F_{\eta}(\eta) ]}
-{\mathbf{E}[\xi \eta]
\over \mathbf{E}[\eta^2]} .
 \end{equation*}
Next, let $\xi _k=X_k-\mathbf{E}[X]$ and $\eta_k=X_k-\mathbf{E}[Y]$. With $\widehat{F}_{\eta}(y)=(n+1)^{-1}\sum_{k=1}^{n} \mathbf{1}\{\eta_k\le y\}$, define
\[
 \widetilde{\Delta}_n=\widetilde{\beta}_{G,n}-\widetilde{\beta}_n,
\]
where
\[
\widetilde{\beta}_{G,n} =\frac{\sum_{k=1}^{n}\xi _k
 w_0\circ \widehat{F}_{\eta}(\eta_k) }{\sum_{k=1}^{n}\eta_k
 w_0\circ \widehat{F}_{\eta}(\eta_k) }
\]
and
\[
\widetilde{\beta}_n= \frac{\sum_{k=1}^{n}\xi _k\eta_k }{\sum_{k=1}^{n}\eta_k^2}.
\]
The following theorem implies that the asymptotic distribution of $n^{1/2}(\widehat{\Delta}_n-\Delta)$ is the same as that of  $n^{1/2}(\widetilde{\Delta}_n-\Delta)$, whose asymptotic normality will be established in Section \ref{sect-4centered} below.

\begin{theorem}
\label{theorem-41}
Suppose that either 1) the conditions of Theorem~\ref{theorem-2} are satisfied  or 2) the conditions of Theorem~\ref{theorem-1} are satisfied and the function $w$ is Lipschitz on the interval $[0,1]$. Then, when $n\to\infty$, we have
\begin{gather}
n^{1/2}(\widehat{\beta}_n - \widetilde{\beta}_n)=o_{\mathbf{P}}(1),
\label{beta_n_as}
\\
n^{1/2}(\widehat{\beta}_{G,n} -\widetilde{\beta}_{G,n})=o_{\mathbf{P}}(1),
\label{b-new}
\end{gather}
and thus
$n^{1/2}(\widehat{\Delta}_n-\Delta) =n^{1/2}(\widetilde{\Delta}_n-\Delta)+o_{\mathbf{P}}(1) $.
\end{theorem}

\begin{proof}
To prove statement (\ref{beta_n_as}), we start with the equations
\begin{align}
\label{c_beta_n}
 \widehat{\beta}_n
 &=\frac{\sum_{k=1}^{n}(\xi _k-(\overline{X}-\mathbf{E}[X]))
 (\eta_k-(\overline{Y}-\mathbf{E}[Y]))}
 {\sum_{k=1}^{n}(\eta_k-(\overline{Y}-\mathbf{E}[Y]))^2 }
 \notag
 \\
 &=\frac{n^{-1}\sum_{k=1}^{n}\xi _k\eta_k
 -(\overline{X}-\mathbf{E}[X])(\overline{Y}-\mathbf{E}[Y])}
 {n^{-1}\sum_{k=1}^{n}\eta_k^2 -(\overline{Y}-\mathbf{E}[Y])^2 }
 \notag
 \\
 &=\frac{\widetilde{\beta}_n-R_{n,1}}{1-R_{n,2}},
 \end{align}
 where
\[
\label{r-1-2}
R_{n,1}=\frac{(\overline{X}-\mathbf{E}[X])(\overline{Y}-\mathbf{E}[Y])}
{n^{-1}\sum_{k=1}^{n}\eta_k^2}
\]
and
\[
R_{n,2}
=\frac{(\overline{Y}-\mathbf{E}[Y])^2}{n^{-1}\sum_{k=1}^{n}\eta_k^2}.
\]
Since $\mathbf{E}[X]$ is finite, the strong law of large numbers implies $\overline{X}\stackrel{a.s.}{\to} \mathbf{E}[X] $ when $n\to\infty$. Furthermore, since the second moment $\mathbf{E}[Y^2]$ is finite, we have  $n^{-1}\sum_{k=1}^{n}\eta_k^2\stackrel{a.s.}{\to} \mathbf{Var}[Y]$. The classical central limit theorem implies
$n^{1/2}(\overline{Y}-\mathbf{E}[Y])=O_{\mathbf{P}}(1)$. Hence, both $R_{n,1}$ and $R_{n,2}$ are of the order $o_{\mathbf{P}}(n^{-1/2})$ when $n\to \infty $, and so equation (\ref{c_beta_n}) implies statement (\ref{beta_n_as}).

We now prove statement (\ref{b-new}). For this, we first note that $\widehat{F}_Y(Y_k)$ is equal to $\widehat{F}_{\eta}(\eta_k)$ and  $w\circ \widehat{F}_Y(Y_k)-\overline{z}_0$ is equal to $ w_0\circ \widehat{F}_{\eta}(\eta_k)-(\overline{z}_0-z_0)$. Hence,
\begin{align*}
\widehat{\beta}_{G,n}
&=\frac{\sum_{k=1}^{n}(\xi _k-(\overline{X}-\mathbf{E}[X]))
( w_0\circ \widehat{F}_{\eta}(\eta_k))-(\overline{z}_0-z_0))}
{\sum_{k=1}^{n}(\eta_k-(\overline{Y}-\mathbf{E}[Y])) ( w_0\circ \widehat{F}_{\eta}(\eta_k))-(\overline{z}_0-z_0))}
\\
&=\frac{n^{-1}\sum_{k=1}^{n}\xi _k w_0\circ \widehat{F}_{\eta}(\eta_k))
-(\overline{X}-\mathbf{E}[X])(\overline{z}_0-z_0)}
{n^{-1}\sum_{k=1}^{n}\eta_k  w_0\circ \widehat{F}_{\eta}(\eta_k)) -(\overline{Y}-\mathbf{E}[Y]) (\overline{z}_0-z_0) },
\end{align*}
and so we have the equation
\begin{equation}
\label{c_beta_Gn}
\widehat{\beta}_{G,n} =\frac{\widetilde{\beta}_{G,n}-R_{n,3}}{1-R_{n,4}}
\end{equation}
 where
\[
R_{n,3}=\frac{(\overline{X}-\mathbf{E}[X])(\overline{z}_0-z_0)}
{n^{-1}\sum_{k=1}^{n}\eta_k  w_0\circ \widehat{F}_{\eta}(\eta_k)}
\]
and
\[
R_{n,4}
=\frac{(\overline{Y}-\mathbf{E}[Y])(\overline{z}_0-z_0)  }{n^{-1}\sum_{k=1}^{n}\eta_k  w_0\circ \widehat{F}_{\eta}(\eta_k)}.
\]
Equation \eqref{c_beta_Gn} implies statement (\ref{b-new}) provided that both $R_{n,3}$ and $R_{n,4}$ are of the order $o_{\mathbf{P}}(n^{-1/2})$, which we prove next. Statement $R_{n,4}=o_{\mathbf{P}}(n^{-1/2})$ follows from the following three facts:
\begin{itemize}
\item
$\overline{Y}-\mathbf{E}[Y]=O_{\mathbf{P}}(n^{-1/2})$, which holds by the central limit theorem;
\item
$\overline{z}_0-z_0\stackrel{a.s.}{\to} 0$, which we already proved under the conditions of either Theorem \ref{theorem-1} or Theorem \ref{theorem-2} (or both);
\item
$n^{-1}\sum_{k=1}^{n}\eta_k  w_0\circ \widehat{F}_{\eta}(\eta_k) \stackrel{a.s.}{\to}  \mathbf{Cov}[Y,w\circ F_Y(Y)]\neq 0 $,
which holds due to statements \eqref{s1} and \eqref{s2}, which we already established under the conditions of either Theorem \ref{theorem-1} or Theorem \ref{theorem-2} (or both).
\end{itemize}

The rest of the proof concerns $R_{n,3}$. When the conditions of Theorem~\ref{theorem-2} are satisfied, in which case the second moment $\mathbf{E}[X^2]$ is finite, the central limit theorem implies $\overline{X}-\mathbf{E}[X]=O_{\mathbf{P}}(n^{-1/2})$ and thus $R_{n,3}=o_{\mathbf{P}}(n^{-1/2})$ in an analogous way as that for $R_{n,4}$. Suppose now that the conditions of  Theorem~\ref{theorem-1} are satisfied, in which case we only have the finiteness of the first moment $\mathbf{E}[X]$, which does not allow us to use the central limit theorem for $\overline{X}$; we can only use $\overline{X}-\mathbf{E}[X]\stackrel{a.s.}{\to} 0$ when $n\to \infty$. However, we can now rely on the additional assumption that the weight function $w$ is Lipschitz on $[0,1]$, and so to establish $R_{n,3}=o_{\mathbf{P}}(n^{-1/2})$, we need to show $\overline{z}_0-z_0=O_{\mathbf{P}}(n^{-1/2})$, whose proof we start with the equation
\begin{equation}
\label{r_3 case_2}
\overline{z}_0-z_0=  {1\over n}\sum_{k=1}^{n}\big( w\circ F_Y(Y_k)-z_0\big) +{1\over n}\sum_{k=1}^{n}\big ( w\circ \widehat{F}_Y(Y_k) - w\circ F_Y(Y_k)\big ).
\end{equation}
Since $w$ is bounded, the first average on the right-hand side of equation~\eqref{r_3 case_2} is of the order $O_{\mathbf{P}}(n^{-1/2})$ due to the classical central limit theorem. For the second average to be of the same order, we use the assumption that $w$ is Lipschitz  on $[0,1]$ and write the bounds
\begin{align}
\label{Dvoretzky}
{1\over n}\Big|\sum_{k=1}^{n}\big ( w\circ \widehat{F}_Y(Y_k) - w\circ F_Y(Y_k)\big )\Big|
&\leq c\,{1 \over n}
\sum_{k=1}^{n}\big| \widehat{F}_Y(Y_k) -  F_Y(Y_k)\big|
\notag \\
&\leq c \,\sup_{y\in \mathbb{R}}\big|\widehat{F}_Y(y) -  F_Y(y)\big|.
 \end{align}
The supremum on the right-hand side of equation~\eqref{r_3 case_2}, usually called the Kolmogorov statistic and denoted by $D_n$, is of the order $O_{\mathbf{P}}(n^{-1/2})$. This completes the proof of statement~\eqref{b-new} and thus establishes Theorem \ref{theorem-41}.
\end{proof}

\subsubsection{Asymptotic normality in the case of centered rv's}
\label{sect-4centered}

Let $\xi $ and $\eta $ be two arbitrary random variables such that $\mathbf{E}[\xi ]=0$ and $\mathbf{E}[\eta ]=0$, and let $w_0$ be a function such that $\mathbf{E}[ w_0\circ F_\eta (\eta )]=0$. (As special cases, we may think of $\xi $, $\eta $ and $w_0$ as those introduced in Section \ref{sect-4transition}.) The counterpart of $\Delta$ within the current context is
\[
\delta = \frac{a}{b}-\frac{c}{d},
\]
where $a =\mathbf{E}[\xi  w_0\circ F_\eta (\eta )]$, $b=\mathbf{E}[\eta  w_0\circ F_\eta (\eta )]$, $c =\mathbf{E}[\xi  \eta ]$, and $d=\mathbf{E}[\eta ^2]$. Hence, in the context of the present section, the weighted-Gini beta is $\beta_G=a/b$ and the classical beta is $\beta=c/d$. The empirical estimator of $\delta $ is
\[
\delta_n = \frac{a_n}{b_n} -\frac{c_n}{d_n},
\]
where $a_n =n^{-1}\sum_{k=1}^{n}\xi _k w_0\circ \widehat{F}_\eta (\eta _k)$, $b_n =n^{-1}\sum_{k=1}^{n}\eta _k w_0\circ \widehat{F}_\eta (\eta _k)$,
$c_n =n^{-1}\sum_{k=1}^{n}\xi _k \eta _k$, and $d_n =n^{-1}\sum_{k=1}^{n} \eta _k^2$. Following our earlier notation, we now work with the conditional-mean function $g_{\xi\mid \eta}(y)=\mathbf{E}[\xi \mid \eta =y]$ and the conditional-variance function $v^2_{\xi\mid \eta}(y)=\textbf{Var}[\xi \mid \eta=y]$. Furthermore, $\eta_{1:n}, \dots, \eta_{n:n}$ denote the (ascending) order statistics of $\eta_1, \dots, \eta_n$, and $\xi_{[1:n]}, \dots , \xi_{[n:n]}$ are the corresponding induced order statistics.
Our next theorem establishes asymptotic normality of $n^{1/2}(\delta_n-\delta)$.

\begin{theorem}
\label{thm_2}

Let the cdf $F_\eta$ be continuous, and let the function $w_0$ be  continuously differentiable except possibly at a finite number of points at which  $F^{-1}_\eta$ and $g_{\xi \mid \eta}\circ F^{-1}_\eta$ must be continuous. Furthermore, assume that there is $b\in [0,1)$ such that the following three assumptions hold:
\begin{enumerate}[{\rm (i)}]
\item  \label{as-2a}
there is $c<\infty $ such that, for all $t\in (0,1)$ where $w'_0(t)$ exists,
\[
|w_0(t)|, ~ t(1-t)|w'_0(t)| \leq c \big ( t(1-t) \big )^{-b/2} ;
\]
\item  \label{as-2b}
the moment $\mathbf{E}[|\eta|^{r_1}] $ is finite for some  $r_1>r:=\max\{4,2/(1-b)\}$;
\item \label{as-2c}
there are $\epsilon>0$ and $c<\infty $ such that, for all $t\in (0,1)$,
\[
v^2_{\xi\mid \eta}\circ F_{\eta}^{-1}(t)\leq  c \big ( t(1-t) \big )^{-2/r+\epsilon},
\]
where we have $2/r=\min\{1/2, 1-b\}$.
\end{enumerate}
Then, when $n \to \infty$,
\[
n^{1/2}(\delta_n-\delta)\stackrel{law}{\to}\mathcal{N}(0,\sigma_1^2+\sigma^2_{2}),
\]
where
\begin{equation}
\label{sigma_1}
\sigma_1^2=\int_0^1v^2_{\xi\mid \eta}\circ  F_\eta ^{-1}(t) \bigg( {w_0(t)\over b}  -{F_\eta ^{-1}(t)\over d}  \bigg)^2\, \mathrm{d}t
\end{equation}
and
\begin{multline}
\sigma_2^2={1\over b^2} \int_0^1 \int_0^1  w_0(s) w_0(t)\big (\min(s,t) -st\big )\,\mathrm{d}h_1(s)\,\mathrm{d}h_1(t)
\\
-{2\over bd} \int_0^1 \int_0^1  w_0(s)\big (\min(s,t) -st\big )\,\mathrm{d}h_1(s)\,\mathrm{d}h_2(t)
\\
+{1\over d^2} \int_0^1 \int_0^1 \big (\min(s,t) -st\big )\,\mathrm{d}h_2(s)\,\mathrm{d}h_2(t)
\label{sigma_2}
\end{multline}
with the functions
\[
h_1(t)=g_{\xi\mid \eta}\circ  F_\eta ^{-1}(t)-\beta_G F_\eta ^{-1}(t)
\]
and
\[
h_2(t)=\big (g_{\xi\mid \eta}\circ  F_\eta ^{-1}(t)-\beta F_\eta ^{-1}(t) \big )F_\eta ^{-1}(t).
\]
\end{theorem}

\begin{proof}
We start with the representation
\begin{align*}
n^{1/2}(\delta_n-\delta)
&= n^{1/2}\left(\frac{a_n}{b_n} -\frac{a}{b}-\frac{c_n}{d_n}+\frac{c}{d}\right)
\\
&= n^{1/2}\left( \frac{1}{b_n}\left( a_n -\frac{b_n}{b}a\right)-\frac{1}{d_n}\left( c_n-\frac{d_n}{d}c\right)\right).
 \end{align*}
Consequently, since $b_n\stackrel{\mathbf{P}}{\to} b$ and $d_n\stackrel{\mathbf{P}}{\to} d$ when $n \to \infty$, the asymptotic distribution of $ n^{1/2}(\delta_n-\delta)$ is the same as that of
\begin{align}
L_n:=&n^{1/2}\left( \frac{1}{b}\left( a_n -\frac{b_n}{b}a\right)-\frac{1}{d}\left( c_n-\frac{d_n}{d}c\right)\right)
\notag
\\
=& n^{1/2}\left( \frac{1}{b}\bigl( a_n -\beta_G b_n\bigr)-\frac{1}{d}\bigl( c_n-\beta d_n\bigr)\right).
\label{clt_15}
\end{align}
Hence, we next prove that the asymptotic distribution of $L_n$ is $\mathcal{N}(0,\sigma_1^2+\sigma^2_{2})$ and in this way complete the proof of Theorem \ref{thm_2}.

Since the cdf $F_{\eta}$ is continuous, $ w_0(\widehat{F}_\eta (\eta _{k:n}))$ is equal to $w_{0,k,n}:=w_0(k/(n+1))$, and so
\begin{align*}
L_n&\stackrel{law}{=}\frac 1{n^{1/2}}\left( {1\over b}\sum_{k=1}^{n}\Big( \xi _{[k:n]} -\beta_G\,\eta _{k:n}\Big) w_{0,k,n} - {1\over d} \sum_{k=1}^{n}\Big( \xi _{[k:n]}\eta _{k:n}-\beta\,\eta _{k:n}^2 \Big) \right)
\\
&=W_n + T_n,
\end{align*}
where
\begin{equation}
\label{term-w}
W_n=\frac 1{n^{1/2}}\sum_{k=1}^{n} \Bigl(\xi _{[k:n]} -g_{\xi\mid \eta}(\eta _{k:n})\Bigr) \bigg( {1\over b}  w_{0,k,n} - {1\over d} \eta _{k:n} \bigg)
\end{equation}
and
\begin{align}
T_n=&\frac 1{n^{1/2}}\sum_{k=1}^{n} g_{\xi\mid \eta}(\eta _{k:n}) \bigg( {1\over b} w_{0,k,n} - {1\over d} \eta _{k:n} \bigg) -
\frac 1{n^{1/2}}\sum_{k=1}^{n}\bigg( {1\over b} \beta_G \eta _{k:n}  w_{0,k,n} -{1\over d} \beta \eta _{k:n}^2\bigg)
\notag
\\
=&\frac 1{n^{1/2}}\sum_{k=1}^{n}\bigg({1\over b}  w_{0,k,n}\Bigl(g_{\xi\mid \eta}(\eta _{k:n})-\beta_G \eta _{k:n}\Bigl) -{1\over d} \Bigl(g_{\xi\mid \eta}(\eta _{k:n})-\beta \eta _{k:n}\Bigl)\eta _{k:n}\bigg).
\label{term-t}
\end{align}
Consequently, we need to show that, when $n \to \infty$,
the random sum $W_n +T_n$ is asymptotically $\mathcal{N}(0,\sigma_1^2+\sigma^2_{2})$, and we rely on Yang (1981) when establishing this result. Namely, we first show (Lemma \ref{yang-1} in Section \ref{auxlemmas} below) that the distribution of $W_n$ conditioned on $\boldsymbol{\eta }_n $ tends to $\mathcal{N}(0,\sigma_1^2)$ for almost all sequences  $(\eta _m)_{m\ge 1}$, and the limiting distribution does not depend on the sequence $(\eta _m)_{m\ge 1}$. Then we prove (Lemma \ref{yang-2} in Section \ref{auxlemmas} below) that $T_n$  is asymptotically $\mathcal{N}(0,\sigma^2_{2})$. Given these two results, we use the following   lemma to conclude that the joint distribution of $(W_n,T_n)$ converges to the product of the two aforementioned normal distributions.  As a special case of this joint convergence of $(W_n,T_n)$, we conclude that $W_n+T_n$ is asymptotically $\mathcal{N}(0,\sigma_1^2+\sigma^2_{2})$, as claimed above. This finishes the proof of Theorem~\ref{thm_2}
\end{proof}

\begin{lemma}[Yang, 1981]
\label{lem_yang}
Let $(\xi _1,\eta _1),(\xi _2,\eta _2),\dots$ be a sequence of random pairs and, for every $n\geq 1$, the first $n$ pairs $(\xi _1,\eta _1),\dots,(\xi _n,\eta _n)$ possess a joint distribution. Denote $\boldsymbol{\zeta}_n= ((\xi _1,\eta _1),\dots,(\xi _n,\eta _n))$ and $\boldsymbol{\eta}_n= (\eta _1,\dots,\eta _n)$, and let $\textbf{W}_n(\boldsymbol{\zeta}_n)$ and $\textbf{T}_n(\boldsymbol{\eta}_n)$ be measurable vector-valued functions of $\boldsymbol{\zeta}_n$ and $\boldsymbol{\eta}_n$, respectively. Suppose that the asymptotic distribution of $\textbf{T}_n$ is $F_T$, and the conditional distribution of $\textbf{W}_n$ given $\boldsymbol{\eta}_n$ is $F_W$, which is assumed not to depend on the $\eta _k$'s. Then the asymptotic distribution of $(\textbf{W}_n,\textbf{T}_n)$ is the product $F_W F_T$.
\end{lemma}

\subsubsection{Two auxiliary lemmas}
\label{auxlemmas}

In this section we deal with $W_n$ defined by equation (\ref{term-w}) and also with $T_n$ defined by equation (\ref{term-t}).

\begin{lemma}
The distribution of $W_n$ conditioned on $\boldsymbol{\eta }_n $ converges to $\mathcal{N}(0,\sigma_1^2)$ for almost all sequences  $(\eta _m)_{m\ge 1}$, and the limiting distribution does not depend on the sequence $(\eta _m)_{m\ge 1}$.
\label{yang-1}
\end{lemma}

\begin{proof}
Using Bhattacharya's (1974) result already utilized in the proof of Theorem \ref{theorem-2}, we have
$\mathbf{E}[W_n \mid \boldsymbol{\eta }_n]=0$ with the conditional variance $V_n^2:=\textbf{Var}[W_n\mid \boldsymbol{\eta }_n]$ expressed by
\[
V_n^2=\frac 1n\sum_{k=1}^{n}v^2_{\xi\mid \eta}(\eta _{k:n})\bigg({1\over b}  w_{0,k,n}-{1\over d} \eta _{k:n}\bigg)^2.
\]
Lindeberg's normal-convergence criterion implies that, when $n\to \infty $, the conditional distribution of $W_n/V_n$ is asymptotically $\mathcal{N}(0,1)$ if, for every $\varepsilon>0$ and when $n\to \infty $,
\begin{equation}
\label{clt_20}
\frac 1{nV_n^2}\sum_{k=1}^{n}\bigg ({1\over b} w_{0,k,n} -{1\over d} \eta _{k:n}\bigg)^2 h_{\theta_{k,n}}(\eta _{k:n}) \rightarrow 0
\end{equation}
for almost all realizations of the sequence $(\eta _m)_{m\ge 1}$, where
\begin{equation}
\label{clt_21}
h_{\theta_{k,n}}(y)=\int (x-g_{\xi\mid \eta}(y))^2 \mathbf{1}\{|x-g_{\xi\mid \eta}(y)|\geq \theta_{k,n}\} \mathrm{d}F(x\mid y)
\end{equation}
with the notation
\[
\theta_{k,n}=\varepsilon n^{1/2} V_n / \left| {1\over b}w_{0,k,n} -{1\over d} \eta _{k:n} \right|.
\]
Under conditions (\ref{as-2a})--(\ref{as-2c}), the strong law of large numbers for $L$-statistics (van Zwet, 1980) implies
\begin{equation}
\label{clt_23}
V_n^2  \stackrel{a.s.}{\to} \int_0^1v^2_{\xi\mid \eta}\circ  F_\eta ^{-1}(t) \bigg ( {1\over b} w_0(t)-{1\over d} F_\eta ^{-1}(t) \bigg )^2\, \mathrm{d}t .
\end{equation}
Note that the integral on the right-hand side is equal to \ $\sigma_1^2 $. To verify  $\theta_{k,n}\stackrel{a.s.}{\to}\infty$, we write the bounds
\begin{align*}
\theta_{k,n}
&\geq \varepsilon n^{1/2} V_n / \left( {1\over b} \max_{k=1,\dots,n}| w_{0,k,n}|+{1\over d} \max_{k=1,\dots,n}|\eta _k| \right)
\notag
\\
&\geq \varepsilon n^{1/2} V_n / \left( {c\over b} \max_{k=1,\dots,n}
\bigg ( {k(n-k+1)\over (n+1)^2} \bigg )^{-b/2} +{1\over d} \max_{k=1,\dots,n}|\eta _k| \right)
\notag
\\
&= \varepsilon n^{1/2} V_n / \left( {c\over b}
(n+1)^{b/2} +{1\over d} \max_{k=1,\dots,n}|\eta _k| \right),
\end{align*}
where we used assumption (\ref{as-2a}). Since $b<1$, we have $b/2<1/2$. Furthermore, since there are at least two finite moments of $\eta $, we have $\max_{k=1,\dots,n}|\eta _k|/n^{1/2}\stackrel{a.s.}{\to}0$. Consequently, $\theta_{k,n}\stackrel{a.s.}{\to} \infty $ as required. Applying the strong law of large numbers for $L$-statistics (van Zwet, 1980), we have, for every $K>0$ and when $n\to \infty $,
\[
\frac 1{n}\sum_{k=1}^{n}\bigg({1\over b}  w_{0,k,n} -{1\over d} \eta _{k:n}\bigg)^2 h_{K}(\eta _{k:n}) \stackrel{a.s.}{\to}
\int_0^1 \bigg({1\over b}  w_0(t)-{1\over d} F_\eta ^{-1}(t)\bigg)^2 h_K\circ F_\eta ^{-1}(t)\, \mathrm{d}t.
\]
The function $h_K(y)$ is defined by equation (\ref{clt_21}), but now with $K$ instead of $\theta_{k,n}$. The just established statement together with statement \eqref{clt_23} verify Lindeberg's criterion for almost all realizations of the sequence $(\eta _m)_{m\ge 1}$ and hence the conditional distribution of $W_n/V_n$ given $\boldsymbol{\eta }_n$ converges to $\mathcal{N}(0,1)$ almost surely. This concludes the proof of Lemma \ref{yang-1}.
\end{proof}

\begin{lemma}
The distribution of $T_n$ converges to $\mathcal{N}(0,\sigma^2_{2})$ when $n\to \infty $.
\label{yang-2}
\end{lemma}

\begin{proof}
We use the approach of Shorack (1972) to tackle $T_n$. Namely, let $U_1,U_2,\dots$ be a sequence of independent $(0,1)$-uniformly distributed random variables, and let $U_{1:n},\dots,U_{n:n}$, denote the order statistics based on the first $n$ members of the sequence. Then we have
\[
T_n\stackrel{law}{=} {1\over b} T_{n,1}-{1\over d}T_{n,2},
\]
where
\[
T_{n,1}=\frac 1{n^{1/2}}\sum_{k=1}^{n} w_{0,k,n} h_1(U_{k:n})
\]
and
\[
T_{n,2}=\frac 1{n^{1/2}}\sum_{k=1}^{n} h_2(U_{k:n}).
\]
Note that $\mathbf{E}[ w_0(U)h_1(U)]=0$ and $\mathbf{E}[h_2(U)]=0$ with $U$ denoting the $(0,1)$-uniformly distributed random variable. Hence, by the strong law of large numbers, we have $n^{-1/2} T_{n,1} \stackrel{a.s.}{\to} 0$ and  $n^{-1/2} T_{n,2} \stackrel{a.s.}{\to} 0$ when $n\to \infty $. We next prove that the asymptotic distribution of $T_n -n^{1/2}\mu_n/b $ is $\mathcal{N}(0,\sigma^2_{2})$, where
\[
\mu_n:=\int_0^1 w_n(t)h_1(t)\, \mathrm{d}t
\]
with $w_n(t) = w_{0,k,n}$ for all $t\in ((k-1)/n, k/n]$ and every $1\leq k\leq n$, and $ w_n(t)=w_{0,1,n}$ when $t=0$.

Let $\mathcal{B}$ be the special Brownian bridge that appears in Shorack~(1972), defined on a special probability space; the change of probability spaces does not affect the current proof. Imitating the form of $T_n$, we define $S$ by the equation
\[
S= \frac{1}{b} S_1 -\frac{1}{d}S_2 ,
\]
where
$S_1= \int_0^1  w_0(t)\mathcal{B}(t)\, \mathrm{d}h_1(t)$
and
$S_2=\int_0^1 \mathcal{B}(t)\, \mathrm{d}h_2(t)$.
Next we write the equation
\begin{equation*}
\label{clt_30}
\bigg ( T_n -\frac{n^{1/2}}b\mu_n \bigg ) -S={1\over b} \Big ( (T_{n,1}-n^{1/2}\mu_n)-S_1\Big ) -{1\over d} \big ( T_{n,2} -S_2 \big ).
\end{equation*}
Repeating the arguments of Shorack (1972), both $(T_{n,1}-n^{1/2}\mu_n)-S_1$ and $T_{n,2} -S_2$ converge to $0$ in probability. Hence, the asymptotic distribution of $T_n -n^{1/2}\mu_n/b $ must be $\mathcal{N}(0,\sigma^2_{2})$ because the distribution of $S$ is $\mathcal{N}(0,\sigma^2_{2})$. We conclude our considerations with the note that, due to assumption~(\ref{as-2a}) on the derivative $ w_0'$,  the quantity $\mu_n$ in $T_n -n^{1/2}\mu_n/b $ can (cf. Shorack, 1972, p.~416) be replaced by $\int_0^1  w_0(t)h_1(t)\, \mathrm{d}t$, which is equal to $0$.

Hence, in summary, the asymptotic distribution of $T_n $ is $\mathcal{N}(0, \sigma^2_{2})$, and this concludes the proof of Lemma \ref{yang-2}.
\end{proof}

\section{Summary}
\label{summary}

In this paper we have developed an asymptotic theory that enables to statistically assess the magnitude of the difference $\Delta $ between the classical and weighted-Gini betas. The former beta has played a~pivotal role in the construction of financial portfolios for several decades, whereas the latter beta has recently arisen in the context of insurance portfolios and risk-capital allocations. Specifically, in this paper we have constructed an estimator for $\Delta $ and derived its consistency and asymptotic normality, which can be used for constructing confidence intervals for, and hypothesis tests about, the difference $\Delta $.

\section*{Acknowledgement}

Research of the second author has been supported by the Natural Sciences and Engineering Research Council of Canada.

\section*{References}
\def\hang{\hangindent=\parindent\noindent}

\hang
Athreya, K.B. (1987).
Bootstrap of the mean in the infinite variance case.
Annals of Statistics, 15, 724--731.

\hang
Bhattacharya, P.K. (1974).
Convergence of sample paths of normalized sum of induced order statistics.
Annals of Statistics, 2, 1034--1039.

\hang
Bickel, P.J., G\"otze, F. and van Zwet, W.R. (1997).
Resampling fewer than  $n$ observations: gains, losses, and remedies for losses.
Statistica Sinica, 7, 1--31.

\hang
Brahimi, B., Meraghni, D., Necir, A. and Zitikis, R. (2011).
Estimating the distortion parameter of the proportional-hazard premium for heavy-tailed losses.
Insurance: Mathematics and Economics, 49, 325--334.

\hang
Chernick, M.R. and LaBudde, R.A. (2011).
An Introduction to Bootstrap Methods with Applications to R.
Wiley, New York.

\hang
Donaldson, D. and  Weymark, J.A. (1980).
A single-parameter generalization of the Gini indices of inequality.
Journal of Economic Theory,  22, 67--86.

\hang
Furman, E. and Zitikis, R. (2008).
Weighted risk capital allocations.
Insurance: Mathematics and Economics, 43, 263-269.

\hang
Furman, E. and Zitikis, R. (2009).
Weighted pricing functionals with application to insurance: an overview.
North American Actuarial Journal, 13, 483--496.

\hang
Furman, E. and Zitikis, R. (2010).
General Stein-type covariance decompositions with
applications to insurance and finance.
ASTIN Bulletin, 40, 369--375.

\hang
Furman, E. and Zitikis, R. (2017).
Beyond the Pearson correlation: heavy-tailed risks, weighted Gini correlations, and a Gini-type weighted insurance pricing model.
ASTIN Bulletin, 47, 919--942.

\hang
Gribkova, N.V. and  Helmers, R. (2007).
On the Edgeworth expansion and the $M$ out of  $N$ bootstrap accuracy  for a Studentized trimmed mean.
Mathematical Methods of Statistics, 16, 142--176.

\hang
Gribkova, N.V. and Helmers, R. (2011).
On the consistency of the $M \ll N$ bootstrap approximation for a trimmed mean.
Theory of Probability and Its Applications, 55, 42--53.

\hang
Hall, P. (1992).
The Bootstrap and Edgeworth Expansion.
Springer, New York.

\hang
Jones, B.L. and Zitikis, R. (2007).
Risk measures, distortion parameters, and their empirical estimation.
Insurance: Mathematics and Economics,
41, 279--297.

\hang
Kakwani, N.C. (1980).
Income Inequality and Poverty: Methods of Estimation and Policy Applications.
Oxford University Press, New York.

\hang
Levy, H. (2010).
The CAPM is alive and well: a review and synthesis.
European Financial Management, 16, 43--71.

\hang
Levy, H. (2011).
The Capital Asset Pricing Model in the 21st Century:
Analytical, Empirical, and Behavioral Perspectives.
Cambridge University Press, Cambridge.

\hang
McNeil, A.J., Frey, R. and Embrechts, P. (2015).
Quantitative Risk Management: Concepts, Techniques and Tools. (Revised Edition.)  Princeton University Press, Princeton, NJ.

\hang
Mammen, E. (1992).
When Does Bootstrap Work? Asymptotic Results and Simulations.
Springer, New York.

\hang
Necir, A. and Meraghni, D. (2009). Empirical estimation of the proportional hazard premium for heavy-tailed claim amounts. Insurance: Mathematics and  Economics,  45, 49--58.

\hang
Necir, A., Meraghni, D., Meddi, F. (2007). Statistical estimate of the proportional hazard premium of loss. Scandinavian Actuarial Journal,  2007, 147--161.

\hang
Petrov, V.V. (1975).
Sums of Independent Random Variables.
Springer, New York.

\hang
Quiggin, J. (1993).
Generalized Expected Utility Theory.
Kluwer, Dordrecht.

\hang
Rao, C.R. and Zhao, L.C. (1995).
Convergence theorems for empirical cumulative quantile regression function.
Mathematical Methods of Statistics, 4, 81--91.

\hang
Sandstr\"{o}m, A. (2010).
Handbook of Solvency for Actuaries and Risk Managers: Theory and Practice.
Chapman and Hall/CRC, Boca Raton, FL.

\hang
Sawyer, N. (2012).
Basel III: Addressing the Challenges of Regulatory Reform.
Risk Books, London.

\hang
Shorack, G.R. (1972).
Functions of order statistics.
Annals of Mathematical Statistics, 43, 412--427.

\hang
Su, J. (2016).
Multiple Risk Factors Dependence Structures with Applications to Actuarial Risk Management. Ph.D.~Thesis, York University, Ontario, Canada.

\hang
Su, J. and Furman, E. (2017). A form of multivariate Pareto
distribution with applications to financial risk measurement.
ASTIN Bulletin, 47, 331--357.

\hang
Tse, S.M. (2009).
On the cumulative quantile regression process.
Mathematical Methods of Statistics, 18, 270--279.

\hang
Tse, S.M. (2015).
The cumulative quantile regression function with censored and truncated response.
Mathematical Methods of Statistics, 24, 147--155.

\hang
van Zwet, W.R. (1980).
A strong law for linear functions of order statistics.
Annals of Probability, 8, 986--990.

\hang
von Neumann, J. and Morgenstern, O. (1944).
Theory of Games and Economic Behavior.
Princeton University Press, Princeton, NJ.

\hang
Wang, S. (1995).
Insurance pricing and increased limits ratemaking by proportional hazards transforms.
Insurance: Mathematics and Economics, 17, 43--54.

\hang
Yang, S.S. (1981).
Linear combinations of concomitants of order statistics with application to testing and estimation.
Annals of the Institute of Statistical Mathematics, 33, 463--470.

\hang
Zitikis, R. and Gastwirth, J.L. (2002).
Asymptotic distribution of the S-Gini index.
Australian and New Zealand Journal of Statistics, 44, 439--446.

\end{document}